\documentclass[11pt]{article}

\usepackage{geometry}
\usepackage{amsthm, amsmath, amsfonts, amssymb, mathtools, mathrsfs}
\usepackage{url, hyperref}
\usepackage{xcolor}
\usepackage{booktabs}
\usepackage{graphicx}

\newtheorem{theorem}{Theorem}[section]
\newtheorem*{theorem*}{Theorem}
\newtheorem{lemma}[theorem]{Lemma}
\newtheorem{proposition}[theorem]{Proposition}

\theoremstyle{remark}
\newtheorem{definition}[theorem]{\bfseries Definition}
\newtheorem{remark}[theorem]{\bfseries Remark}

\newtheorem{conjecture}[theorem]{\bfseries Conjecture}
\newtheorem*{acknowledgements}{\bfseries Acknowledgements}

\makeatletter
\newcommand{\subjclass}[2][2020]{%
	\let\@oldtitle\@title%
	\gdef\@title{\@oldtitle\footnotetext{#1 \emph{Mathematics subject classification}: #2}}%
}
\newcommand{\keywords}[1]{%
	\let\@@oldtitle\@title%
	\gdef\@title{\@@oldtitle\footnotetext{\emph{Key words and phrases}: #1}}%
}
\makeatother
\renewcommand{\ge}{\geqslant}
\renewcommand{\le}{\leqslant}

\newcommand{\bbZ}{\mathbb{Z}}
\newcommand{\bbQ}{\mathbb{Q}}

\newcommand{\bi}{\mathbf{i}}

\newcommand{\rd}{\mathrm{d}}

\numberwithin{equation}{section}

\allowdisplaybreaks

\title{The linear independence of $1$, $\zeta(3)$, and $\zeta(4)$}
\author{Li Lai}
\date{}


\begin{document}
\maketitle

\begin{abstract}
In a recent breakthrough, Fauzan proved the irrationality of $\zeta(5)$.
In this paper, we combine Fauzan's determinant method with classical simultaneous rational approximations to zeta values studied by Zudilin and Dauguet--Zudilin to prove the linear independence of $1$, $\zeta(3)$, and $\zeta(4)$ over $\mathbb{Q}$. 
As a corollary, $\zeta(3)/\pi^4$ is irrational.
\end{abstract}

\section{Introduction}

The special values of the Riemann zeta function
\[
\zeta(s) \coloneq \sum_{m=1}^{\infty} \frac{1}{m^s} \quad (s=2,3,\ldots)
\]
are natural mathematical constants.
They appear frequently in number theory, combinatorics, and theoretical physics.
In the eighteenth century, Euler \cite{Euler1740} proved that $\zeta(2k)/\pi^{2k} \in \bbQ^{\times}$ for any positive integer $k$.
In contrast, the arithmetic nature of \emph{odd zeta values} $\zeta(2k+1)$ ($k \in \bbZ_{>0}$) remains mysterious.
A folklore conjecture predicts the following.
\begin{conjecture}\label{conj}
$\pi,\zeta(3),\zeta(5),\ldots,\zeta(2k+1),\ldots$ are algebraically independent over $\bbQ$.
\end{conjecture}
Conjecture \ref{conj} remains far out of reach;
very little is known.
Ap\'ery made the first breakthrough: 
\begin{theorem}[Ap\'ery \cite{Apery1979}, 1979]\label{thm_Ap}
The real number $\zeta(3)$ is irrational.
\end{theorem}
Ap\'ery’s announcement in 1978 came as a surprise. 
A few months later, Beukers \cite{Beukers1979} gave an elegant proof of the irrationality of $\zeta(3)$, which is perhaps the most widely known one.
Nowadays, many other proofs and explanations of Ap\'ery's theorem (Theorem \ref{thm_Ap}) are available; see \cite{Beukers1981,Beukers1987,BP1984,Brown2016,DZ2014,Fischler2004,Gutnik1979,Gutnik1983,Hata2000,KR2007,Nesterenko1996,Nesterenko2009,vanderPoorten1979,Prevost1996,Reyssat1979,RV2001,SoriaLorente2014,Sorokin1994,Sorokin1998,Sorokin2002,Zudilin2009,Zudilin2017}.

The next breakthrough was made by Ball and Rivoal:

\begin{theorem}[Rivoal \cite{Rivoal2000}, 2000; Ball--Rivoal \cite{BR2001}, 2001]\label{thm_BR}
As $s \to +\infty$ through odd integers, 
\[
\dim_{\bbQ}\operatorname{Span}_\bbQ \{ 1,\zeta(3),\zeta(5),\ldots,\zeta(s) \} \ge \frac{1-o(1)}{1+\log 2} \cdot \log s.
\]
In particular, there are infinitely many irrational odd zeta values.
\end{theorem}
The dimension lower bound in Theorem \ref{thm_BR} has been significantly improved to $0.21\sqrt{s/\log s}$ by Fischler \cite{Fischler2026} for all sufficiently large odd integers $s$.

Before the AI age, a remarkable result was achieved by Calegari, Dimitrov, and Y. Tang:
\begin{theorem}[Calegari--Dimitrov--Tang \cite{CDT2024+}, 2024$^+$]
The three real numbers $1$, $\zeta(2)$, and $L(2,\chi_{-3})$ are linearly independent over $\bbQ$, where $L(2,\chi_{-3})=1^{-2} -2^{-2} + 4^{-2} - 5^{-2} + \cdots$ is a Dirichlet $L$-value.
\end{theorem}

Although there are many partial results towards the irrationality of odd zeta values \cite{BZ2022+,Fischler2017,FSZ2018,FSZ2019,FZ2010,Lai2024+,Lai2026,LY2020,LZ2022,Rivoal2002,RZ2020,SoriaLorente2018,Sprang2018+,Zudilin2001,Zudilin2002a,Zudilin2002b,Zudilin2004,Zudilin2018}, the irrationality of $\zeta(5)$ remained open before the AI age.
The closest results to the irrationality of $\zeta(5)$ were ``at least one of $\zeta(5),\zeta(7),\zeta(9),\zeta(11)$ is irrational'' by Zudilin \cite{Zudilin2001}, and ``$0<|\zeta(5)-p_n/q_n|<q_n^{-0.86}$ for an explicit rational sequence $p_n/q_n$'' by Brown--Zudilin \cite{BZ2022+}.
In the AI age, Fauzan recently made the breakthrough: 

\begin{theorem}[Fauzan \cite{Fauzan2026+}, 2026$^+$]
The real number $\zeta(5)$ is irrational.
\end{theorem}

Fauzan’s proof uses Hankel determinants, which were previously studied by Haynes--Zudilin \cite{HZ2015} and Zudilin \cite{Zudilin2017}. 
Fauzan's proof is also related to some ideas of Brown \cite{Brown2026}. 
Roughly speaking, Hankel determinants yield finer $p$-adic estimates. 

The following is a simple explanation of Fauzan's key observation.
Let us first consider a toy model: the Vandermonde determinant $I_n = \det(i^{j})_{0 \le i,j < n}$. 
For a prime $p$, considering each entry gives the trivial lower bound on the $p$-adic valuations: $v_p(i^{j}) \ge 0$, so $v_p(I_{n}) \ge 0$. 
In fact, $v_p(I_n)$ can be much larger than $0$ when $n$ is large.
In the literature, there are many known linear combinations of $1$ and $\zeta(5)$ with rational coefficients, of the form $a_{i,j} + b_{i,j}\zeta(5)$. 
These linear combinations do not suffice to prove the irrationality of $\zeta(5)$.
Fauzan arranged a suitable sequence of such linear forms into a Hankel determinant $\widetilde{I}_n = \det(a_{i,j} + b_{i,j}\zeta(5))_{0 \le i, j < n}$. 
Using basic properties of determinants, and assuming that $\zeta(5) \in \mathbb{Q}$, Fauzan was able to show that $v_p(\widetilde{I}_n)$ is much larger than the entrywise lower bound $n \cdot \min v_p(a_{i,j} + b_{i,j}\zeta(5))$.
This is the key reason why Fauzan's proof \cite{Fauzan2026+} works.

Fauzan's determinant method is quite flexible. 
It quickly led to several further results: 
\begin{itemize}
\item Kilani \cite{K2026+} proved the irrationality of the $2$-adic zeta value $\zeta_2(7)$;
\item Liu \cite{Liu2026+} formalized a proof in Lean 4 of the irrationality of $\operatorname{Li}_2(1/n)$ for any $n \in \bbZ\setminus\{0,1\}$;
\item Q. Tang \cite{Tang2026+} formalized a proof in Lean 4 of the linear independence of $1$, $\zeta(2)$, and $\zeta(3)$ over $\mathbb{Q}$;
\item OpenAI \cite{OAI2026+} formalized a proof in Lean 4 of the irrationality of Catalan's constant $L(2,\chi_{-4})$;
\item Kilani \cite{K2026+b} formalized a proof in Lean 4 of the irrationality of the $3$-adic zeta value $\zeta_3(5)$.
\end{itemize}
The author apologizes for any possible omissions. 
In the AI age, many recent results have not appeared on arXiv, which makes it difficult to keep track of all of them.

In this paper, we obtain another follow-up result.
Our main theorem is the following:
\begin{theorem}\label{thm_main}
The three real numbers $1$, $\zeta(3)$, and $\zeta(4)$ are linearly independent over $\bbQ$.
\end{theorem}
The proof combines Fauzan's determinant method \cite{Fauzan2026+} with simultaneous rational approximations to zeta values studied by Zudilin \cite{Zudilin2002b,Zudilin2007,Zudilin2011,Zudilin2014} and Dauguet--Zudilin \cite{DZ2014}. 
After Fauzan's breakthrough, such a quick result is more like an exercise, as commented by Zudilin and agreed by the author.
Nevertheless, the author decided to write this paper to help himself understand Fauzan’s determinant method. 
As we shall see, the method is elementary, and
the determinant structure strengthens the $p$-adic estimates.

Perhaps the most interesting corollary of Theorem \ref{thm_main} is the irrationality of $\zeta(3)/\pi^4$. 
Note that the irrationality of $\zeta(3)/\pi^2$ follows from \cite{Tang2026+}. 
However, to the best of the author’s knowledge, the irrationality of $\zeta(3)/\pi^3$ remains open.

We give an outline of the proof of Theorem \ref{thm_main}.
Suppose that $a \zeta(3)+b \zeta(4)=c$ for integers $a, b, c$, not all zero. 
In Section \ref{sec_2}, we associate to this relation a sequence of rational Hankel determinants $J_n$. 
Their rationality follows from a partial fraction decomposition of the underlying rational function. 
In Section \ref{sec_3}, an auxiliary-prime argument shows that $J_n \neq 0$ for infinitely many $n$. 
In Section \ref{sec_4}, to estimate the denominators, we group the poles by their residue classes modulo each prime and estimate minors of the resulting local matrices. 
Their rank restrictions and divisibility give substantially stronger bounds than an entrywise estimate. 
After summation over the primes in Section \ref{sec_5}, we obtain 
\[
\sum_{p} v_p(J_n)\log p \ge -(30.83+o(1)) n^2 \quad\text{as } n \to \infty.
\]
On the other hand, in Section \ref{sec_6}, a contour representation, Heine's identity, and orthogonal polynomials give
\[
\left|J_n\right| \le e^{-31.61 \cdot n^2}
\] 
for all sufficiently large $n$. 
These bounds contradict each other along the infinite nonzero subsequence.

\medskip

\noindent\textbf{Declaration of generative AI use}: The author sent the following prompt to ChatGPT (GPT-6 Astra, OpenAI) on October 1, 2026:
\begin{quotation}
Let $R_n(t)=\frac{(t-Bn)_{Bn}^3}{(t)_{An+1}^2}$, and let $\Delta_n$ be the determinant of size $Cn$, whose $(i,j)$-entry is a suitable linear combination of $\sum_{m=1}^{\infty} \frac{\mathrm{d}}{\mathrm{d}t}\left(R_n(t)t^{i+j}\right) |_{t=m}$ and $\sum_{m=1}^{\infty} \frac{\mathrm{d}^2}{\mathrm{d}t^2}\left(R_n(t)t^{i+j}\right) |_{t=m}$, where $A,B,C$ are positive integers such that $3B+2C \le 2A$. 
Can you use this type of construction to prove the linear independence of $1,\zeta(3),\zeta(4)$?
\end{quotation}
After approximately 40 minutes, GPT-6 Astra produced a complete proof.
Then, the author verified and rewrote the proof, and takes full responsibility but no credit for this work.
The author also used DeepSeek (DeepSeek-V4.1-Flash, DeepSeek) to check the grammar.

\medskip

\begin{acknowledgements}
The author thanks Wadim Zudilin for many useful discussions and encouragement.
This work was supported by the Research Foundation for Scholars of Xiamen University (Grant No. X2450218). 
\end{acknowledgements}

\section{Rational functions and determinants}\label{sec_2}

Our proof is by contradiction.
If $1$, $\zeta(3)$, and $\zeta(4)$ are linearly dependent over $\mathbb{Q}$, then there exist integers $a$, $b$, and $c$, not all zero, such that $a\zeta(3)+b\zeta(4)=c$.
By Ap\'ery's theorem (Theorem \ref{thm_Ap}), we must have $b \neq 0$.
Throughout the paper, until we finally obtain a contradiction, we fix a choice of such integers $a$, $b$, and $c$. 
For later reference, we record our assumption below:
\begin{equation}\label{def_abc}
a\zeta(3) + b\zeta(4) = c, \quad a,b,c \in \bbZ, \quad b \neq 0.
\end{equation}

We use $\bi \coloneq \sqrt{-1} \in \mathbb{C}$ to denote the imaginary unit.
The letter $i$ will be frequently used as an integer index.  
Let $t$ denote an indeterminate. 
The Pochhammer symbol is defined by $(t)_n \coloneq t(t+1)\cdots(t+n-1)$ for any $n \in \bbZ_{>0}$, with the convention $(t)_0 \coloneq 1$. 
For any rational function $R(t)=P(t)/Q(t) \in \bbQ(t)$ with $P,Q \in \bbQ[t]$, $Q \neq 0$, define the degree of $R(t)$ to be $\deg R \coloneq \deg P - \deg Q$.
We adopt the convention $\deg 0 \coloneq -\infty$.
For a field $K$, denote by $K[t]_{<n}$ the $K$-linear space $\{ P \in K[t] \mid \deg P < n\}$ for any $n \in \bbZ_{>0}$.

\begin{definition}
For any $n \in \bbZ_{>0}$, define the rational function
\begin{equation}\label{def_Rn(t)}
R_n(t) \coloneq \frac{(6n)!^2}{(2n)!^3} \cdot \frac{(t-2n)_{2n}^3}{(t)_{6n+1}^2} \in \bbQ(t).
\end{equation}
Define the linear operator
\begin{align}
\mathcal{L}_n \colon \mathbb{C}[t]_{<6n} &\longrightarrow \mathbb{C}, \notag \\ 
P &\mapsto  -3a \sum_{m=1}^{\infty} \frac{\rd}{\rd t} \Big( R_n(t)P(t) \Big) \Big|_{t=m} + b \sum_{m=1}^{\infty} \frac{\rd^2}{\rd t^2} \Big( R_n(t)P(t) \Big) \Big|_{t=m}  \label{def_Lambda}.
\end{align}
Since $\deg (R_nP) \le -3$, both series on the right-hand side of \eqref{def_Lambda} converge absolutely in $\mathbb{C}$.
Define the symmetric bilinear form 
\begin{align}
\mathcal{B}_n \colon \bbQ[t]_{<3n} \times \bbQ[t]_{<3n} &\longrightarrow \mathbb{R}, \notag \\
(P,Q) &\mapsto \mathcal{L}_n(P \cdot Q). \label{def_Qn}
\end{align}
Note that the matrix of $\mathcal{B}_n$ with respect to the basis $1,t,t^2,\ldots,t^{3n-1}$ of $\bbQ[t]_{<3n}$ is $\Big(\mathcal{L}_n(t^{i+j})\Big)_{0 \le i,j < 3n}$. 
Define
\begin{equation}\label{def_Jn}
J_n \coloneq \frac{1}{\prod_{i=0}^{3n-1} i!^{2}} \cdot \det\Big(\mathcal{B}_n(t^{i},t^{j})\Big)_{0 \le i,j < 3n} = \frac{1}{\prod_{i=0}^{3n-1} i!^{2}} \cdot \det\Big(\mathcal{L}_n(t^{i+j})\Big)_{0 \le i,j < 3n} \in \mathbb{R}.
\end{equation}
Here $1/\prod_{i=0}^{3n-1} i!^{2}$ is a conventional normalization factor.
\end{definition}

As usual, let $H_k^{(s)} \coloneq \sum_{m=1}^{k} m^{-s}$ denote the $k$th generalized harmonic number of order $s$ for any $k,s \in \bbZ_{>0}$. 
By convention, $H_0^{(s)} \coloneq 0$ for any $s \in \bbZ_{>0}$.
Write $H_k \coloneq H_k^{(1)}$ for $k \in \bbZ_{\ge 0}$.

\begin{lemma}\label{lem_explicit_expression_for_Lambda_n(P)}
For any $P \in \bbQ[t]_{<6n}$, we have
\begin{equation}\label{eqn_explicit_expression_for_Lambda_n(P)}
\mathcal{L}_n(P) = \sum_{k=0}^{6n} \binom{6n}{k}^2\binom{2n+k}{k}^3\Big( U_k \cdot P^\prime(-k) + V_k \cdot P(-k) \Big) \in \bbQ,
\end{equation}
where 
\begin{align}
U_k &= -3aH_{k}^{(2)} - 2b H_{k}^{(3)} \in \bbQ, \label{def_Uj} \\
V_k &= \left( 3H_{2n+k} + 2H_{6n-k} - 5H_k \right)(3aH_{k}^{(2)} + 2b H_{k}^{(3)}) +6\left( c - aH_k^{(3)} - bH_k^{(4)} \right) \in \bbQ. \label{def_Vj}
\end{align}
In particular, $J_n \in \bbQ$.
\end{lemma}

\begin{proof}
Since the rational function $R_n(t)P(t)$ has $\deg(R_nP) \le -3$ and all poles lie in $\{0,-1,\ldots,-6n\}$ and have order $\le 2$, there is a unique partial fraction decomposition of the form
\begin{equation}\label{eqn_2.9}
R_n(t)P(t) = \sum_{k=0}^{6n} \left( \frac{A_{1,k}(P)}{t+k} + \frac{A_{2,k}(P)}{(t+k)^2} \right),
\end{equation}
where
\begin{align}
A_{2,k}(P) &= (t+k)^2R_n(t)P(t)\Big|_{t=-k} =  \binom{6n}{k}^2\binom{2n+k}{k}^3P(-k), \label{def_A2j(P)}\\
A_{1,k}(P) &= \frac{\rd}{\rd t}\left( (t+k)^2R_n(t)P(t) \right)\Big|_{t=-k} \notag\\
&= \binom{6n}{k}^2\binom{2n+k}{k}^3 \left( P^\prime(-k) + \left( \sum_{j=1}^{2n} \frac{3}{-k-j} - \sum_{\substack{0 \le j \le 6n \\ j \neq k}} \frac{2}{-k+j} \right) P(-k) \right) \notag \\
&= \binom{6n}{k}^2\binom{2n+k}{k}^3 \Big( P^\prime(-k) + \left( 5H_k -3H_{2n+k} -2H_{6n-k} \right) P(-k) \Big). \label{def_A1j(P)}
\end{align}
By \eqref{eqn_2.9} and $\deg R_nP \le -3$, we have
\begin{equation}\label{eqn_rho1=0}
\sum_{k=0}^{6n} A_{1,k}(P) = \lim_{t \to \infty} t R_n(t)P(t) = 0.
\end{equation}

Substituting \eqref{eqn_2.9} into \eqref{def_Lambda}, and using \eqref{eqn_rho1=0} and \eqref{def_abc}, we obtain
\begin{align*}
\mathcal{L}_n(P) =&\ 3a \sum_{k=0}^{6n} \left( A_{1,k}(P) \sum_{m=1}^{\infty} \frac{1}{(m+k)^2} +2 A_{2,k}(P) \sum_{m=1}^{\infty} \frac{1}{(m+k)^3} \right) \\
&\ + b\sum_{k=0}^{6n} \left( 2A_{1,k}(P) \sum_{m=1}^{\infty} \frac{1}{(m+k)^3} +6 A_{2,k}(P) \sum_{m=1}^{\infty} \frac{1}{(m+k)^4} \right) \\
=&\ 3a \sum_{k=0}^{6n}  \left( A_{1,k}(P) \left( \zeta(2) - H_k^{(2)} \right) +2 A_{2,k}(P) \left(\zeta(3)- H_k^{(3)} \right) \right) \\
&\ + b\sum_{k=0}^{6n} \left( 2A_{1,k}(P) \left( \zeta(3) - H_k^{(3)} \right) +6 A_{2,k}(P) \left( \zeta(4) - H_k^{(4)} \right) \right) \\
=&\ \sum_{k=0}^{6n} \left( \left( -3aH_{k}^{(2)} -2bH_{k}^{(3)} \right) A_{1,k}(P) + 6\left( a\zeta(3)+b\zeta(4) -aH_{k}^{(3)} - bH_{k}^{(4)} \right)A_{2,k}(P) \right) \\
=&\ \sum_{k=0}^{6n} \left( \left( -3aH_{k}^{(2)} -2bH_{k}^{(3)} \right) A_{1,k}(P) + 6\left( c -aH_{k}^{(3)} - bH_{k}^{(4)} \right)A_{2,k}(P) \right).
\end{align*}
Finally, substituting \eqref{def_A2j(P)} and \eqref{def_A1j(P)} into the right-hand side of the above equation, we obtain \eqref{eqn_explicit_expression_for_Lambda_n(P)}.
Therefore, each entry of the determinant in \eqref{def_Jn} is a rational number.
Thus, $J_n \in \bbQ$.
\end{proof}

\section{Nonvanishing of the determinants}\label{sec_3}

In this section, we prove that $J_n \neq 0$ infinitely often by using auxiliary primes $q$. 
This method appeared before in the study of $p$-adic zeta values by Sprang and the author \cite{LS2023+}, and Kilani \cite{K2026+}. 

\begin{proposition}\label{prop_nonva}
Let $q$ be a prime number such that $q \equiv 1 \pmod{9}$ and $q > |b|$. 
Set $m = (q-1)/9$. 
Then $J_{2m} \neq 0$.
In particular, there are infinitely many positive integers $n$ such that $J_n \neq 0$.
\end{proposition}

\begin{proof}
By \eqref{def_Jn} with $n=2m$, we have
\[
J_{2m} = \frac{1}{\prod_{i=0}^{6m-1} i!^{2}} \cdot \det\Big(\mathcal{B}_{2m}(t^{i},t^{j})\Big)_{0 \le i,j < 6m}.
\]
For $0 \le k \le 3m-1$, define the polynomial and constant
\[
F_k(t) \coloneq \prod_{\substack{0 \le \ell \le 3m-1 \\ \ell \neq k}} \frac{t+\ell}{\ell-k} \in \bbQ[t], \quad a_k \coloneq F_k^{\prime}(-k) = \sum_{\substack{0 \le \ell \le 3m-1 \\ \ell \neq k}} \frac{1}{\ell-k} \in \bbQ.
\]
Put 
\[
P_k(t) =q^2 \cdot (1 - 2a_k \cdot (t+k)) \cdot F_k(t)^2, \quad Q_k(t)=q \cdot (t+k) \cdot F_k(t)^2, \quad 0 \le k \le 3m-1.
\]
These polynomials have degree at most $6m-1$, and
\begin{equation}\label{5.1}
P_k(-\ell) = q^2\delta_{k,\ell}, \quad P_k^{\prime}(-\ell) = 0, \quad Q_k(-\ell) = 0, \quad Q_k^{\prime}(-\ell) = q\delta_{k,\ell}, \quad 0 \le k, \ell \le 3m-1,
\end{equation}
where $\delta_{k,\ell} \in \{0,1\}$ is Kronecker's delta.
Therefore, $P_0,Q_0,P_1,Q_1,\ldots,P_{3m-1},Q_{3m-1}$ form a basis of $\bbQ[t]_{<6m}$.
To prove $J_{2m} \neq 0$, it suffices to show that the matrix of the bilinear form $\mathcal{B}_{2m}$ with respect to this basis has nonzero determinant.

For any $P,Q \in \bbQ[t]_{<6m}$, by \eqref{def_Qn}, \eqref{def_Lambda}, and Lemma \ref{lem_explicit_expression_for_Lambda_n(P)}, we have
\begin{equation}\label{5.2}
\mathcal{B}_{2m}\left( P,Q \right) = \sum_{j=0}^{12m} \mathcal{B}_{2m}^{(j)}(P,Q),
\end{equation}
where 
\begin{align}
\mathcal{B}_{2m}^{(j)}(P,Q) &= \Big( \widehat{U}_j \cdot P^\prime(-j) \cdot Q(-j) +\widehat{U}_j \cdot P(-j) \cdot Q^{\prime}(-j) + \widehat{V}_j \cdot P(-j) \cdot Q(-j) \Big), \label{5.3} \\
\widehat{U}_j &= -(3aH_{j}^{(2)} + 2b H_{j}^{(3)})\binom{12m}{j}^2\binom{4m+j}{j}^3,  \notag\\
\widehat{V}_j &= \Big( \left( 3H_{4m+j} + 2H_{12m-j} - 5H_j \right)(3aH_{j}^{(2)} + 2b H_{j}^{(3)}) +6\left( c - aH_j^{(3)} - bH_j^{(4)} \right) \Big) \notag\\ & \quad \times \binom{12m}{j}^2\binom{4m+j}{j}^3. \notag
\end{align}

For $0 \le j < q = 9m+1$, we have $v_q(\widehat{U}_{j}) \ge 0$, $v_q(\widehat{V}_{j}) \ge -1$, and $P_k, Q_k \in q\bbZ_q[t]$ ($0 \le k \le 3m-1$). 
Thus, by \eqref{5.3}, we have
\begin{equation}\label{5.4}
\mathcal{B}_{2m}^{(j)}(P_{k_1},P_{k_2}), \mathcal{B}_{2m}^{(j)}(Q_{k_1},Q_{k_2}), \mathcal{B}_{2m}^{(j)}(P_{k_1},Q_{k_2}) \in q\mathbb{Z}_q, \quad 0 \le j \le 9m,\ 0 \le k_1,k_2 \le 3m-1.
\end{equation}

For $j = q+ \ell$, where $0 \le \ell \le 3m-1$, we have $H_j^{(s)} = q^{-s}+O(\bbZ_q)$ ($s=1,2,3,4$), $H_{4m+j}= q^{-1}+O(\bbZ_q)$, $H_{12m-j} \in \bbZ_q$, and
\[
b_j \coloneq \binom{12m}{j}^2\binom{4m+j}{j}^3 \in \bbZ_q^{\times}, \quad 9m+1 \le j=q+\ell \le 12m.
\]
Since $q>|b|$ and $b \neq 0$ by \eqref{def_abc}, we have
\[
b \in \bbZ_q^{\times}.
\]
Thus, 
\begin{equation}\label{5.5}
\widehat{U}_j = -2bb_j q^{-3} + O\left( q^{-2}\bbZ_q \right), \quad \widehat{V}_j = -10bb_j q^{-4} + O\left( q^{-3}\bbZ_q \right).
\end{equation}
On the other hand, by Taylor expansion at $-\ell$ and using \eqref{5.1}, we have
\begin{align}
P_k(-j) &= q^2 \delta_{k,\ell} + O(q^4\bbZ_q), \qquad P_k^{\prime}(-j) = O(q^3\bbZ_q), \label{5.6}\\
Q_k(-j) &= -q^2 \delta_{k,\ell} + O(q^3\bbZ_q), \qquad Q_k^{\prime}(-j) =q\delta_{k,\ell} + O(q^2\bbZ_q). \label{5.7}
\end{align}
Substituting \eqref{5.5}--\eqref{5.7} into \eqref{5.3}, we obtain 
\begin{align}
\mathcal{B}_{2m}^{(j)}(P_{k_1},P_{k_2}) &= -10bb_j \delta_{k_1,\ell}\delta_{k_2,\ell} + O(q\bbZ_q), \label{5.8}\\
\mathcal{B}_{2m}^{(j)}(Q_{k_1},Q_{k_2}) &= -6bb_j \delta_{k_1,\ell}\delta_{k_2,\ell} + O(q\bbZ_q), \label{5.9}\\
\mathcal{B}_{2m}^{(j)}(P_{k_1},Q_{k_2}) &= 8bb_j \delta_{k_1,\ell}\delta_{k_2,\ell} + O(q\bbZ_q), \label{5.10}\\
\text{where } j &= q +\ell, \quad 0 \le \ell \le 3m-1, \quad 0 \le k_1,k_2 \le 3m-1. \notag
\end{align}

Now, let $\boldsymbol{M}_{2m}$ be the matrix of the symmetric bilinear form $\mathcal{B}_{2m}$ with respect to the basis $P_0$, $Q_0$, $\ldots$, $P_{3m-1}$, $Q_{3m-1}$ of $\bbQ[t]_{<6m}$.
By \eqref{5.2}, \eqref{5.4}, and \eqref{5.8}--\eqref{5.10}, we have
\[
\boldsymbol{M}_{2m} \in \operatorname{Mat}_{6m}(\bbZ_q), \quad \boldsymbol{M}_{2m} \equiv \operatorname{diag}\left( \left(\begin{array}{cc}
	-10bb_{j} & 8bb_{j} \\
	8bb_{j} & -6bb_{j}
\end{array}\right) \right)_{9m+1 \le j \le 12m} \pmod{\operatorname{Mat}_{6m}(q\bbZ_q)}.
\]
Therefore, 
\[
\det \boldsymbol{M}_{2m} \equiv \prod_{j=9m+1}^{12m}(-4b^2b_j^2) \pmod{q\bbZ_q}, \quad \det \boldsymbol{M}_{2m} \in \bbZ_q^{\times}.
\]
So $\det \boldsymbol{M}_{2m} \neq 0$ and thereby $J_{2m} \neq 0$.

Dirichlet's theorem \cite{Dirichlet1837} gives infinitely many primes $q \equiv 1 \pmod{9}$. 
Therefore, $J_n \neq 0$ infinitely often.
\end{proof}

\section{\texorpdfstring{$p$}{p}-adic estimates for the determinants}\label{sec_4}

As we have shown, under the assumption \eqref{def_abc}, we have $J_n \in \bbQ$. 
In this section, we will give a lower bound for the $p$-adic order of $J_n$ prime-by-prime. 
For any prime $p$, denote by $\log_p$ the logarithm to base $p$. 
As usual, $v_p(x)$ denotes the $p$-adic order for any $x \in \bbQ$, normalized by $v_p(p)=1$ with the convention $v_p(0) = +\infty$. 
We first recall the classical $p$-adic estimates.

\begin{lemma}[classical $p$-adic estimates] \label{lem_cla_p_est}
The following $p$-adic estimates hold.
\begin{enumerate}
\item[\textup{(1)}] For any prime $p$, 
\[
v_p(J_n) \ge -12n\log_p(8n).
\]
\item[\textup{(2)}] For any prime $p > \sqrt{8n}$,
\[
v_p(J_n) \ge -12n.
\]
\item[\textup{(3)}] For any prime $p > 6n$, 
\[
v_p(J_n) \ge 0.
\]
\end{enumerate}
\end{lemma}

\begin{proof}
Consider the new basis $\binom{t}{0}$, $\binom{t}{1}$, $\binom{t}{2}$, $\ldots$, $\binom{t}{3n-1}$ of the space $\bbQ[t]_{<3n}$.
The transfer matrix from the old basis $1,t,t^2,\ldots,t^{3n-1}$ to the new basis is triangular with diagonal entries $1/i!$ ($i=0,1,\ldots,3n-1$).
Therefore, by \eqref{def_Qn} and \eqref{def_Jn}, 
\begin{equation}\label{eqn_3.1}
J_n = \det\left( \mathcal{B}_n\left( \binom{t}{i},\binom{t}{j} \right) \right)_{0 \le i,j <3n} = \det\left( \mathcal{L}_n\left( \binom{t}{i}\cdot\binom{t}{j} \right) \right)_{0 \le i,j <3n}.
\end{equation}
It is elementary (see Zudilin \cite[Lemma 15]{Zudilin2004}) to show that 
\[
\operatorname{lcm}[1,2,\ldots,i] \cdot \frac{\rd}{\rd t}\binom{t}{i}\Big|_{t=-k} \in \bbZ \quad\text{for any $k \in \bbZ$ and $i \in \bbZ_{>0}$}. 
\]
Therefore,
\begin{equation}\label{eqn_3.2}
v_p\left(  \frac{\rd}{\rd t}\binom{t}{i}\Big|_{t=-k} \right) \ge - \lfloor \log_p(3n) \rfloor \quad\text{for any $k \in \bbZ$ and $i \in \{0,1,\ldots,3n-1\}$}. 
\end{equation}
By \eqref{eqn_explicit_expression_for_Lambda_n(P)}, for any $0 \le i,j < 3n$ we have
\begin{align}
&\ \mathcal{L}_n\left( \binom{t}{i}\cdot\binom{t}{j} \right) \notag \\
=&\ \sum_{k=0}^{6n} \binom{6n}{k}^2\binom{2n+k}{k}^3\left( U_k\binom{-k}{i}\frac{\rd}{\rd t}\binom{t}{j}\Big|_{t=-k} + U_k \binom{-k}{j}\frac{\rd}{\rd t}\binom{t}{i}\Big|_{t=-k} + V_k\binom{-k}{i}\binom{-k}{j} \right), \label{eqn_3.3}
\end{align}
where $U_k$ and $V_k$ are given by \eqref{def_Uj} and \eqref{def_Vj}.

Now, for any prime $p$, since $v_p(H_k^{(s)}) \ge -s\lfloor \log_p k \rfloor$ for any $k,s \in \bbZ_{>0}$, we have
\begin{equation}\label{eqn_3.4}
v_p\left( U_k \right) \ge -3\lfloor \log_p(6n) \rfloor, \quad v_p\left(V_k\right) \ge -4\lfloor \log_p(8n) \rfloor, \quad k \in \{0,1,\ldots,6n\}.
\end{equation}
Combining \eqref{eqn_3.2}--\eqref{eqn_3.4}, we obtain
\begin{equation}\label{eqn_3.5}
v_p\left( \mathcal{L}_n\left( \binom{t}{i}\cdot\binom{t}{j} \right) \right) \ge -4\lfloor \log_p(8n) \rfloor  \quad\text{for any prime $p$ and } i,j \in \{0,1,\ldots,3n-1\}.
\end{equation}
By \eqref{eqn_3.1} and \eqref{eqn_3.5}, we obtain
\[
v_p(J_n) \ge -12n \lfloor \log_p(8n) \rfloor,
\]
which proves both (1) and (2).

Finally, for (3), note that for any $p > 6n$ and $k \in \{0,1,\ldots,6n\}$, we have 
\[
v_p\left(\binom{2n+k}{k} H_{2n+k}\right) \ge 0 \Longrightarrow v_p\left(\binom{2n+k}{k} V_k\right) \ge 0. 
\]
Clearly, $v_p(U_k) \ge 0$ and $v_p\left(  \frac{\rd}{\rd t}\binom{t}{i}\big|_{t=-k} \right) \ge 0$ by \eqref{eqn_3.2}. 
Therefore, \eqref{eqn_3.3} implies 
\[
v_p\left( \mathcal{L}_n\left( \binom{t}{i}\cdot\binom{t}{j} \right) \right) \ge 0 \quad\text{for any prime $p>6n$ and } i,j \in \{0,1,\ldots,3n-1\}.
\]
Thus, \eqref{eqn_3.1} implies $v_p(J_n) \ge 0$ for any prime $p>6n$, proving (3).
\end{proof}

In the rest of this section, we will establish finer $p$-adic estimates for primes $\sqrt{8n}< p \le 6n$. 
This is the key advantage of Fauzan's determinant method.
In spirit, these finer $p$-adic estimates can be viewed as a kind of generalization of the $\Phi_n$-factor method of Zudilin \cite{Zudilin2001,Zudilin2004}.

\begin{definition}
For any prime $p$ and $\rho \in \{0,1,\ldots,p-1\}$, define the linear operator
\begin{align}
\mathcal{L}_n^{\rho \bmod p} \colon \bbQ[t]_{<6n} &\longrightarrow \bbQ \notag \\
P & \mapsto \sum_{\substack{ 0 \le k \le 6n \\ k \equiv \rho \bmod p }} \binom{6n}{k}^2\binom{2n+k}{k}^3\Big( U_k \cdot P^\prime(-k) + V_k \cdot P(-k) \Big), \label{def_Lambda_p}
\end{align}
where $U_k$ and $V_k$ are given by \eqref{def_Uj} and \eqref{def_Vj}.
Define the symmetric bilinear form
\begin{align}
\mathcal{B}_n^{\rho \bmod p} \colon \bbQ[t]_{<3n} \times \bbQ[t]_{<3n} &\longrightarrow \bbQ \notag \\
(P,Q) &\mapsto \mathcal{L}_n^{\rho \bmod p}(P \cdot Q). \label{def_Qn_p}
\end{align}
Define the functions $F^{\rho \bmod p}$ on nonnegative integers:
\begin{align}
F^{\rho \bmod p}(m) \coloneq&\ m^2 +  \left( 2\left\lfloor \frac{6n}{p} \right\rfloor -  2\left\lfloor \frac{6n-\rho}{p} \right\rfloor + 3\left\lfloor \frac{2n+\rho}{p} \right\rfloor - 3\left\lfloor \frac{2n}{p} \right\rfloor - 5 \right) \cdot m 
\notag\\
&\ + 3 \cdot \max\left( 0, m - 2\left\lfloor \frac{6n-\rho}{p} \right\rfloor \right) + \mathbf{1}_{m=1} \cdot \mathbf{1}_{p>\max\{6n-\rho,2n+\rho\}}. \label{def_F1} 
\end{align}
\end{definition}

\begin{lemma}\label{lem_fin_p_est}
For any prime $\sqrt{8n}< p \le 6n$ and $\rho \in \{0,1,\ldots,p-1\}$, for any subsets $I,J \subset \{0,1,\ldots,3n-1\}$ of size $|I|=|J|=m \in \bbZ_{\ge 0}$, we have
\begin{equation}\label{eqn_fin_p_est}
v_p\left( \det\Big( \mathcal{B}_n^{\rho \bmod p}(t^i,t^j) \Big)_{i \in I, j \in J} \right) \ge F^{\rho \bmod p}(m).
\end{equation}
When $m > 2 + 2\lfloor (6n-\rho)/p \rfloor$, the determinant on the left-hand side of \eqref{eqn_fin_p_est} vanishes.
When $m=0$, the empty determinant is understood as $1$.
\end{lemma}

\begin{proof}
Let $\boldsymbol{E}_{\rho}$ be the transfer matrix from the basis $1$, $(t+\rho)$, $(t+\rho)^2$, $\ldots$, $(t+\rho)^{3n-1}$ of $\bbQ[t]_{<3n}$ to the basis $1,t,t^2,\ldots,t^{3n-1}$. 
Then $\boldsymbol{E}_{\rho}$ is a unitriangular matrix with integral entries.
Thus, $\boldsymbol{E}_{\rho} \in \operatorname{GL}_{3n}(\bbZ)$, and
\begin{equation}\label{eqn_3.8}
\Big( \mathcal{B}_n^{\rho \bmod p}(t^i,t^j) \Big)_{0 \le i,j < 3n} = \boldsymbol{E}_{\rho} \cdot \Big( \mathcal{B}_n^{\rho \bmod p}\big( (t+\rho)^i,(t+\rho)^j \big) \Big)_{0 \le i,j < 3n} \cdot \boldsymbol{E}_{\rho}^\top.
\end{equation}
For convenience, write 
\[
c_k =  \binom{6n}{k}^2\binom{2n+k}{k}^3, \quad k \in \{0,1,\ldots,6n\}.
\]
Since $p>\sqrt{8n}$, we have
\[
v_p\left( c_k \right) = 2\left\lfloor \frac{6n}{p} \right\rfloor - 2 \left\lfloor \frac{k}{p} \right\rfloor - 2\left\lfloor \frac{6n-k}{p} \right\rfloor + 3\left\lfloor \frac{2n+k}{p} \right\rfloor - 3\left\lfloor \frac{k}{p} \right\rfloor - 3\left\lfloor \frac{2n}{p} \right\rfloor.
\]
Therefore,
\begin{equation}\label{eqn_3.9}
v_p(c_k) = 2\left\lfloor \frac{6n}{p} \right\rfloor -  2\left\lfloor \frac{6n-\rho}{p} \right\rfloor + 3\left\lfloor \frac{2n+\rho}{p} \right\rfloor - 3\left\lfloor \frac{2n}{p} \right\rfloor \quad\text{if } k \equiv \rho \bmod p.
\end{equation}
Note that $0 \le k \le 6n$ and $k \equiv \rho \bmod p$ is equivalent to $k=\rho+\ell p$ for some $0 \le \ell \le \lfloor (6n-\rho)/p \rfloor$.
Thus, by \eqref{def_Qn_p} and \eqref{def_Lambda_p}, we have
\begin{align*}
\mathcal{B}_n^{\rho \bmod p}\big( (t+\rho)^i,(t+\rho)^j \big) = \sum_{\ell=0}^{\lfloor (6n-\rho)/p \rfloor} c_{\rho+\ell p} \left( U_{\rho+\ell p} \cdot \frac{\rd}{\rd t}(t+\rho)^{i+j}\Big|_{t=-\rho-\ell p} + V_{\rho+\ell p} \cdot (t+\rho)^{i+j}\Big|_{t=-\rho-\ell p} \right).
\end{align*}
Therefore, we have the following identity of matrices
\begin{equation}\label{eqn_3.10}
\Big( \mathcal{B}_n^{\rho \bmod p}\big( (t+\rho)^i,(t+\rho)^j \big) \Big)_{0 \le i,j < 3n} = \boldsymbol{A}_{0}^{\rho \bmod p} + \boldsymbol{A}_{\ge 1}^{\rho \bmod p},
\end{equation}
where
\begin{align}
\boldsymbol{A}_{0}^{\rho \bmod p} =&\ \Big( c_{\rho} \left( U_{\rho} \delta_{i,1}\delta_{j,0} + U_{\rho} \delta_{i,0}\delta_{j,1} + V_{\rho} \delta_{i,0}\delta_{j,0} \right) \Big)_{0 \le i,j < 3n} \in \operatorname{Mat}_{3n}(\bbQ) \label{def_A0}, \\
\boldsymbol{A}_{\ge 1}^{\rho \bmod p} =&\ \sum_{\ell=1}^{\lfloor (6n-\rho)/p \rfloor} \Big( c_{\rho+\ell p} \left( \frac{U_{\rho+\ell p}}{-\ell p} \cdot i + V_{\rho+\ell p} \right) (-\ell p)^{i+j} \Big)_{0 \le i,j < 3n} \notag\\
&\ + \sum_{\ell=1}^{\lfloor (6n-\rho)/p \rfloor} \Big( c_{\rho+\ell p} \cdot \frac{U_{\rho+\ell p}}{-\ell p} \cdot j \cdot (-\ell p)^{i+j} \Big)_{0 \le i,j < 3n} \in \operatorname{Mat}_{3n}(\bbQ). \label{def_A1}
\end{align}
If $\lfloor (6n-\rho)/p \rfloor= 0$, then $\boldsymbol{A}_{\ge 1}^{\rho \bmod p} = \boldsymbol{0}$.
Equation \eqref{def_A1} implies that $\boldsymbol{A}_{\ge 1}^{\rho \bmod p}$ is a sum of $2\lfloor (6n-\rho)/p \rfloor$ matrices of rank $\le 1$, so
\begin{equation}\label{eqn_rank}
\operatorname{rank} \boldsymbol{A}_{\ge 1}^{\rho \bmod p} \le 2 \left\lfloor \frac{6n-\rho}{p} \right\rfloor.
\end{equation}
Clearly, we also have $\operatorname{rank} \boldsymbol{A}_{0}^{\rho \bmod p} \le 2$, and so 
\begin{equation}\label{eqn_rankB}
\operatorname{rank} \Big( \mathcal{B}_n^{\rho \bmod p}\big( (t+\rho)^i,(t+\rho)^j \big) \Big)_{0 \le i,j < 3n} \le 2 + 2 \left\lfloor \frac{6n-\rho}{p} \right\rfloor.
\end{equation}
Equations \eqref{eqn_rankB} and \eqref{eqn_3.8} prove the statement ``when $m > 2 + 2\lfloor (6n-\rho)/p \rfloor$, the determinant on the left-hand side of \eqref{eqn_fin_p_est} vanishes'' in Lemma \ref{lem_fin_p_est}.

By \eqref{def_Uj} and \eqref{def_Vj}, we have $v_p(U_{\rho+\ell p}) \ge -3$ and $v_p(V_{\rho+\ell p}) \ge -4$. 
Combining these $p$-adic estimates with \eqref{eqn_3.9}, we obtain the lower bound for the $p$-adic order of the $(i,j)$th entry of $\boldsymbol{A}_{\ge 1}^{\rho \bmod p}$:
\begin{equation}\label{eqn_3.14}
v_p\left( \left(\boldsymbol{A}_{\ge 1}^{\rho \bmod p}\right)_{i,j} \right) \ge  i+j+2\left\lfloor \frac{6n}{p} \right\rfloor -  2\left\lfloor \frac{6n-\rho}{p} \right\rfloor + 3\left\lfloor \frac{2n+\rho}{p} \right\rfloor - 3\left\lfloor \frac{2n}{p} \right\rfloor - 4.
\end{equation}
Noting that $\rho < p$, by \eqref{def_Uj} and \eqref{def_Vj}, we have $v_p(U_\rho) \ge 0$ and $v_p(V_\rho) \ge -1$.
Therefore, the $p$-adic order of the $(i,j)$th entry of $\boldsymbol{A}_{0}^{\rho \bmod p}$ is better:
\begin{equation}\label{eqn_3.15}
v_p\left( \left(\boldsymbol{A}_{0}^{\rho \bmod p}\right)_{i,j} \right) \ge  i+j+2\left\lfloor \frac{6n}{p} \right\rfloor -  2\left\lfloor \frac{6n-\rho}{p} \right\rfloor + 3\left\lfloor \frac{2n+\rho}{p} \right\rfloor - 3\left\lfloor \frac{2n}{p} \right\rfloor - 1.
\end{equation}
Now, for any subsets $I',J' \subset \{0,1,\ldots,3n-1\}$ of size $m$, expand 
\[
\det \Big( \mathcal{B}_n^{\rho \bmod p}\big( (t+\rho)^i,(t+\rho)^j \big) \Big)_{i \in I', j \in J'} = \det \Big( \boldsymbol{A}_{0}^{\rho \bmod p} + \boldsymbol{A}_{\ge 1}^{\rho \bmod p} \Big)_{I',J'}
\]
multilinearly in its columns to be a sum of several resulting determinants. 
By \eqref{eqn_rank}, if we choose more than $2\lfloor (6n-\rho)/p \rfloor$ columns from $\boldsymbol{A}_{\ge 1}^{\rho \bmod p}$, then the resulting determinant is zero. 
We may only consider the resulting determinants such that at most $2\lfloor (6n-\rho)/p \rfloor$ columns come from $\boldsymbol{A}_{\ge 1}^{\rho \bmod p}$.
Then, by \eqref{eqn_3.14} and \eqref{eqn_3.15}, we obtain
\begin{align*}
&\ v_p\left( \det \Big( \mathcal{B}_n^{\rho \bmod p}\big( (t+\rho)^i,(t+\rho)^j \big) \Big)_{i \in I', j \in J'} \right) \\
\ge&\ \sum_{i \in I'} i + \sum_{j \in J'} j \\
&\ + \left( 2\left\lfloor \frac{6n}{p} \right\rfloor -  2\left\lfloor \frac{6n-\rho}{p} \right\rfloor + 3\left\lfloor \frac{2n+\rho}{p} \right\rfloor - 3\left\lfloor \frac{2n}{p} \right\rfloor - 4 \right) \cdot m \\
&\ + 3 \cdot \max\left\{ 0, m - 2\left\lfloor \frac{6n-\rho}{p} \right\rfloor \right\} \\
\ge&\ F^{\rho \bmod p}(m)
\end{align*}
except for the special case $m=1$ and $p > \max\{6n-\rho,2n+\rho\}$, where the last step used $\sum_{i \in I'} i \ge 0+1+\cdots+(m-1) = m(m-1)/2$, the analogous inequality for $J'$, and \eqref{def_F1}. 

In the special case $m=1$ and $p > \max\{6n-\rho,2n+\rho\}$, we have $\lfloor (6n-\rho)/p \rfloor= 0$ so $\boldsymbol{A}_{\ge 1}^{\rho \bmod p} = \boldsymbol{0}$. 
Since $p > \max\{6n-\rho,2n+\rho\}$, we now have $v_p(U_{\rho}) \ge 0$ and $v_p(V_{\rho}) \ge 0$ (see \eqref{def_Uj} and \eqref{def_Vj}).
Thus, by \eqref{def_A0}, every entry of $\Big( \mathcal{B}_n^{\rho \bmod p}\big( (t+\rho)^i,(t+\rho)^j \big) \Big)_{0 \le i,j < 3n} =\boldsymbol{A}_{0}^{\rho \bmod p}$ is a product of $c_\rho$ and a $p$-adic integer. 
So \eqref{eqn_3.9} implies every entry of $\Big( \mathcal{B}_n^{\rho \bmod p}\big( (t+\rho)^i,(t+\rho)^j \big) \Big)_{0 \le i,j < 3n}$ has $p$-adic order at least 
\[
2\left\lfloor \frac{6n}{p} \right\rfloor -  2\left\lfloor \frac{6n-\rho}{p} \right\rfloor + 3\left\lfloor \frac{2n+\rho}{p} \right\rfloor - 3\left\lfloor \frac{2n}{p} \right\rfloor.
\] 
This is exactly $F^{\rho \bmod p}(1)$ in the special case (see \eqref{def_F1}).
In conclusion, we always have
\begin{equation}\label{eqn_3.19}
v_p\left( \det \Big( \mathcal{B}_n^{\rho \bmod p}\big( (t+\rho)^i,(t+\rho)^j \big) \Big)_{i \in I', j \in J'} \right) \ge F^{\rho \bmod p}(m)
\end{equation}
for any subsets $I',J' \subset \{0,1,\ldots,3n-1\}$ of any size $m$.

Finally, by \eqref{eqn_3.8} and the Cauchy--Binet formula, for any $I,J \subset \{0,1,\ldots,3n-1\}$ of size $m$, we have
\begin{align*}
&\ \det\Big( \mathcal{B}_n^{\rho \bmod p}(t^i,t^j) \Big)_{i \in I, j \in J} \\
=&\ \sum_{\substack{I',J' \subset \{0,1,\ldots,3n-1\} \\ |I'|=|J'|=m}} \det\Big(\boldsymbol{E}_{\rho}\Big)_{I,I'} \det \Big( \mathcal{B}_n^{\rho \bmod p}\big( (t+\rho)^i,(t+\rho)^j \big) \Big)_{i \in I', j \in J'} \det\Big(\boldsymbol{E}_{\rho}^{\top}\Big)_{J',J}.
\end{align*}
By \eqref{eqn_3.19} and $\boldsymbol{E}_{\rho} \in \operatorname{GL}_{3n}(\bbZ)$, we obtain \eqref{eqn_fin_p_est}.
The proof of Lemma \ref{lem_fin_p_est} is complete.
\end{proof}

\begin{lemma}[finer $p$-adic estimates]\label{lem_vpJn}
For any prime $\sqrt{8n}< p \le 6n$, we have
\[
v_p\left( J_n \right) \ge  - \frac{9n^2}{p} + 3n -p \left\{ \frac{3n}{p} \right\}\left( 1 - \left\{ \frac{3n}{p} \right\} \right) + \min_{\substack{ 0 \le m_{\rho} \le 2+2\lfloor (6n-\rho)/p \rfloor \\ \sum_{\rho=0}^{p-1} m_\rho = 3n }} \sum_{\rho=0}^{p-1} F^{\rho \bmod p}(m_\rho),
\]
where $\{x\}=x-\lfloor x \rfloor$ is the fractional part of a real number $x$, and the functions $F^{\rho \bmod p}$ are given by \eqref{def_F1}.
\end{lemma}

\begin{proof}
By \eqref{def_Lambda}, \eqref{def_Qn}, \eqref{def_Lambda_p}, and \eqref{def_Qn_p}, we have
\[
\Big( \mathcal{B}_n(t^i,t^j) \Big)_{0 \le i,j < 3n} = \sum_{\rho=0}^{p-1} \Big( \mathcal{B}_n^{\rho \bmod p}(t^i,t^j) \Big)_{0 \le i,j < 3n}.
\]
Expand its determinant first by multilinearity in the columns and then by Laplace expansion in the rows.
We deduce that $\det \Big( \mathcal{B}_n(t^i,t^j) \Big)_{0 \le i,j < 3n}$ is a signed sum of products
\begin{equation}\label{eqn_3.20}
\prod_{\rho=0}^{p-1} \det\Big( \mathcal{B}_n^{\rho \bmod p}(t^i,t^j) \Big)_{i \in I_{\rho}, j\in J_{\rho}},
\end{equation}
where $\{ I_{\rho} \}_{0 \le \rho < p}$ and $\{ J_{\rho} \}_{0 \le \rho < p}$ partition the row and column indices, and
\[
|I_{\rho}|=|J_{\rho}|=m_{\rho}, \quad \sum_{\rho=0}^{p-1} m_{\rho} = 3n.
\]
By Lemma \ref{lem_fin_p_est}, each factor of \eqref{eqn_3.20} has $p$-adic order at least $F^{\rho \bmod p}(m_\rho)$, and it vanishes if $m_{\rho} > 2+2\lfloor (6n-\rho)/p \rfloor$. 
Therefore,
\begin{equation}\label{eqn_3.21}
v_p\left( \det \Big( \mathcal{B}_n(t^i,t^j) \Big)_{0 \le i,j < 3n} \right) \ge \min_{\substack{ 0 \le m_{\rho} \le 2+2\lfloor (6n-\rho)/p \rfloor \\ \sum_{\rho=0}^{p-1} m_\rho = 3n }} \sum_{\rho=0}^{p-1} F^{\rho \bmod p}(m_\rho).
\end{equation}

Since $p>\sqrt{8n}$, we have
\begin{equation}\label{eqn_3.22}
v_p\left( \prod_{i=0}^{3n-1} i!^2 \right) = 2\sum_{i=0}^{3n-1} \left\lfloor \frac{i}{p} \right\rfloor = \frac{9n^2}{p} - 3n +p \left\{ \frac{3n}{p} \right\}\left( 1 - \left\{ \frac{3n}{p} \right\} \right).
\end{equation}
Now, \eqref{def_Jn}, \eqref{eqn_3.21}, and \eqref{eqn_3.22} prove the lemma.
\end{proof}

\section{Summation over primes}\label{sec_5}

In this section, we sum the lower bounds for $v_p(J_n)$ over primes $p$. 
The lower bounds in Lemma \ref{lem_vpJn} are discrete in nature.
Therefore, we have to consider several ranges of $p$ and carry out a case-by-case analysis. 
Indeed, we shall consider three main ranges.
\[
p \le n; \qquad n < p \le 3n; \qquad 3n < p \le 6n.
\]
For different ranges, we will take different (but related) strategies to address them.
The most important range is $3n<p \le 6n$.
For this case, we need to exploit the full strength of Lemma \ref{lem_vpJn}. 
For the other two ranges, we will sacrifice some of the estimates from Lemma \ref{lem_vpJn} to simplify the summation $\sum v_p(J_n)\log p$.

We need the following lemma, which is a corollary of the prime number theorem.

\begin{lemma}\label{lem_4.1}
Let $v>u>0$ be absolute constants. 
For any bounded piecewise continuous function $f$ on $[u,v]$, we have
\begin{equation}\label{pnt}
\lim_{n \to \infty} \frac{1}{n} \sum_{\substack{un < p \le vn \\ p \text{ prime }}} f\left(\frac{p}{n}\right) \log p = \int_{u}^{v} f(x)\,\rd x.
\end{equation}
\end{lemma}

\begin{proof}
By the prime number theorem \cite{Hadamard1896,ValleePoussin1896}, we have $\sum_{p \le x} \log p \sim x$ as $x \to +\infty$. 
Equation \eqref{pnt} follows in the special case that $f$ is a step function.
The general case then follows by a standard approximation argument. 
\end{proof}

Now, we deal with the first range.

\begin{lemma}\label{lem_4.2}
For any prime $\sqrt{8n} < p \le n$, we have
\[
v_p(J_n) \ge \left( 9\left\{ \frac{2n}{p} \right\} - 6\left\{ \frac{6n}{p} \right\} - 6 \right) n - \frac{29}{16}p - 3 \sqrt{n}.
\]
Consequently,
\[
\liminf_{n \to \infty} \frac{1}{n^2}\sum_{p \le n} v_p(J_n) \log p \ge -5.69.
\]
\end{lemma}

\begin{proof}
For any prime $p$ in the interval $\sqrt{8n} < p \le n$ and $\rho \in \{0,1,\ldots,p-1\}$, let
\[
s_{\rho \bmod p} \coloneq 2\left\lfloor \frac{6n}{p} \right\rfloor -  2\left\lfloor \frac{6n-\rho}{p} \right\rfloor + 3\left\lfloor \frac{2n+\rho}{p} \right\rfloor - 3\left\lfloor \frac{2n}{p} \right\rfloor, \quad \overline{s}_p \coloneq \frac{1}{p} \sum_{\rho=0}^{p-1} s_{\rho \bmod p}.
\]
Noting that
\[
s_{\rho \bmod p} = - 2 \left\lfloor \left\{ \frac{6n}{p} \right\} - \frac{\rho}{p} \right\rfloor + 3 \left\lfloor \left\{ \frac{2n}{p} \right\} + \frac{\rho}{p} \right\rfloor = 2 \cdot \mathbf{1}_{\rho/p > \{6n/p\}} + 3 \cdot \mathbf{1}_{\rho/p \ge 1 -\{2n/p\}},
\]
we can calculate
\begin{equation}\label{4.2}
\overline{s}_p = 2 - 2\left\{ \frac{6n}{p} \right\} - \frac{2}{p} + 3\left\{ \frac{2n}{p} \right\}.
\end{equation}

By \eqref{def_F1}, we have $F^{\rho \bmod p}(m) \ge m^2 + (s_{\rho \bmod p} - 5)m$ for any nonnegative integer $m$.
Therefore, for any nonnegative integers $m_{\rho}$ with $\sum_{\rho=0}^{p-1} m_{\rho} = 3n$, we have
\begin{align*}
\sum_{\rho=0}^{p-1} F^{\rho \bmod p}(m_\rho) &\ge \sum_{\rho=0}^{p-1} \left( m_\rho^2 + (s_{\rho \bmod p}-5)m_\rho \right) \\ 
&= \sum_{\rho=0}^{p-1} \left( m_{\rho} + \frac{s_{\rho \bmod p} - \overline{s}_p}{2} \right)^2 + 3n(\overline{s}_p - 5) + \frac{p}{4} \overline{s}_p^2 - \frac{1}{4}\sum_{\rho=0}^{p-1} s_{\rho \bmod p}^2 \\
&\ge \frac{9n^2}{p} + 3n(\overline{s}_p - 5) + \frac{p}{4} \overline{s}_p^2 - \frac{1}{4}\sum_{\rho=0}^{p-1} 5s_{\rho \bmod p} \\
&= \frac{9n^2}{p} + 3n(\overline{s}_p - 5) - \frac{p}{4}  \overline{s}_p(5-\overline{s}_p) \\
&\ge \frac{9n^2}{p} + 3n(\overline{s}_p - 5) - \frac{25}{16}p,
\end{align*}
where the third-to-last step used Cauchy's inequality and $0 \le s_{\rho \bmod p} \le 5$.
Thus, Lemma \ref{lem_vpJn} implies
\begin{align}
v_p(J_n) &\ge 3n(\overline{s}_p - 4) -p \left\{ \frac{3n}{p} \right\}\left( 1 - \left\{ \frac{3n}{p} \right\} \right) - \frac{25}{16}p \notag\\
&\ge 3n(\overline{s}_p - 4) - \frac{29}{16} p, \label{eqn_4.1}
\end{align}
where the last step used $\theta(1-\theta) \le 1/4$ for $\theta = \{3n/p\}$. 
Substituting \eqref{4.2} into \eqref{eqn_4.1} and using $p>\sqrt{8n}$, we obtain
\begin{equation}\label{4.3}
v_p(J_n) \ge \left( 9\left\{ \frac{2n}{p} \right\} - 6\left\{ \frac{6n}{p} \right\} - 6 \right) n - \frac{29}{16}p - 3 \sqrt{n}.
\end{equation}

Consider the function
\[
f_1(x) \coloneq 9\left\{ \frac{2}{x} \right\} - 6 \left\{ \frac{6}{x} \right\} - 6 - \frac{29}{16}x, \quad x \in (0,1].
\]
For any $\varepsilon  \in (0,1)$, the function $f_1(x)$ is piecewise continuous on $[\varepsilon,1]$. 
By \eqref{4.3} and Lemma \ref{lem_4.1}, we have
\begin{equation}\label{4.4}
\frac{1}{n^2} \sum_{\varepsilon n < p \le n} v_p(J_n) \log p \ge  \int_{\varepsilon}^{1} f_1(x) \,\rd x  - o(1) \quad\text{as } n \to \infty.
\end{equation}
By Lemma \ref{lem_cla_p_est} (1)(2) and the prime number theorem, we have
\begin{align}
\sum_{p \le \sqrt{8n}} v_p(J_n) \log p &\ge -12 \sqrt{8} n^{3/2} \log (8n), \label{4.5} \\
\sum_{\sqrt{8n} < p \le \varepsilon n} v_p(J_n) \log p &\ge -12n \sum_{p \le \varepsilon n} \log p \ge -\left(12\varepsilon +o(1)\right) n^2 \quad\text{as } n \to \infty. \label{4.6}
\end{align}
Putting \eqref{4.4}--\eqref{4.6} together, and letting $\varepsilon \to 0^{+}$, we obtain
\begin{equation}\label{4.7}
\liminf_{n \to \infty} \frac{1}{n^2}\sum_{p \le n} v_p(J_n) \log p \ge \int_0^{1} f_1(x) \,\rd x.
\end{equation}

For any $m \in \bbZ_{>0}$, we have
\[
\int_{0}^{1} \left\{  \frac{m}{x} \right\} \rd x = m \int_{m}^{+\infty} \frac{\{y\}}{y^2} \rd y = m\cdot\left( H_m - \gamma -\log m \right),
\]
where $\gamma = 0.577\ldots$ is the Euler--Mascheroni constant and $H_m=\sum_{k=1}^{m} k^{-1}$. 
Therefore, 
\begin{align*}
\int_0^{1} f_1(x) \,\rd x = 18 \gamma + 18\log 18 - \frac{10897}{160} = -5.689676\ldots,
\end{align*} 
and \eqref{4.7} gives
\[
\liminf_{n \to \infty} \frac{1}{n^2}\sum_{p \le n} v_p(J_n) \log p \ge -5.69.
\]
\end{proof}

Next, we tackle the second range.

\begin{lemma}\label{lem_4.3}
We have
\[
\liminf_{n \to \infty} \frac{1}{n^2}\sum_{n < p \le 3n} v_p(J_n) \log p \ge -9.39.
\]
\end{lemma}

\begin{proof}
Consider a prime $n < p \le 3n$.
We assume $n \ge 10$ so that $p>n>\sqrt{8n}$.
The starting point is the same as that of Lemma \ref{lem_4.2}.
For any $\rho \in \{0,1,\ldots,p-1\}$, by \eqref{def_F1}, we have $F^{\rho \bmod p}(m) \ge m^2 + (s_{\rho \bmod p} - 5)m$ for any $m \in \bbZ_{\ge 0}$, where
\begin{equation}\label{eqn_srho}
s_{\rho \bmod p} =2 \cdot \mathbf{1}_{\rho/p > \{6n/p\}} + 3 \cdot \mathbf{1}_{\rho/p \ge 1 -\{2n/p\}}  \in \{0,2,3,5\}.	
\end{equation}
Roughly speaking, $F^{\rho \bmod p}(m)$ grows at least quadratically in $m$. 
Now, the idea is to find a lower bound for $F^{\rho \bmod p}(m)$ with linear growth in $m$. 
Let $\lambda$ be an integer parameter to be determined. 
For any $m \in \bbZ_{\ge 0}$, we have
\begin{align*}
F^{\rho \bmod p}(m) &\ge \lambda m + \left( m^2 +(s_{\rho \bmod p} - \lambda - 5)m  \right) \\
&\ge \lambda m - \left\lfloor \frac{(\lambda + 5 - s_{\rho \bmod p})^2}{4} \right\rfloor.
\end{align*}
Define
\[
N_s(p) \coloneq \# \left\{  \rho \in \{0,1,\ldots,p-1\} \mid  s_{\rho \bmod p} = s  \right\}, \quad s=0,2,3,5.
\]
Then,
\begin{equation}\label{4.8}
\min_{\substack{ 0 \le m_{\rho} \le 2+2\lfloor (6n-\rho)/p \rfloor \\ \sum_{\rho=0}^{p-1} m_\rho = 3n }} \sum_{\rho=0}^{p-1} F^{\rho \bmod p}(m_\rho) \ge 3\lambda n  - \sum_{s \in \{0,2,3,5\}} N_s(p) \left\lfloor \frac{(\lambda + 5 - s)^2}{4} \right\rfloor.
\end{equation}
Since $n<p \le 3n$, we have
\begin{equation}\label{4.00}
- \frac{9n^2}{p} + 3n -p \left\{ \frac{3n}{p} \right\}\left( 1 - \left\{ \frac{3n}{p} \right\} \right) = \begin{cases}
6p -12n \quad&\text{if } n < p \le \frac{3}{2}n \\
2p -6n \quad&\text{if } \frac{3}{2}n < p \le 3n
\end{cases}.
\end{equation}
Adding \eqref{4.8} and \eqref{4.00}, and using Lemma \ref{lem_vpJn}, we obtain
\begin{equation}\label{4.11}
v_p(J_n) \ge (6p-12n) \cdot \mathbf{1}_{n<p\le 3n/2} + (2p-6n)\cdot \mathbf{1}_{3n/2 < p \le 3n} + 3\lambda n  - \sum_{s \in \{0,2,3,5\}} N_s(p) \left\lfloor \frac{(\lambda + 5 - s)^2}{4} \right\rfloor.
\end{equation}
Now, using \eqref{eqn_srho}, we have
\begin{align*}
N_0(p) &= \min\left( 1+p\left\{\frac{6n}{p}\right\}, p-p\left\{\frac{2n}{p}\right\} \right), \\
N_2(p) &= \max\left( p-1-p\left\{\frac{6n}{p}\right\}-p\left\{\frac{2n}{p}\right\}, 0 \right), \\
N_3(p) &= \max\left( p\left\{\frac{6n}{p}\right\}+p\left\{\frac{2n}{p}\right\} + 1 -p, 0 \right), \\
N_5(p) &= \min\left( p-1-p\left\{\frac{6n}{p}\right\}, p\left\{\frac{2n}{p}\right\} \right).
\end{align*}
Below we consider several ranges of $p$. 
In each range, we can express $N_0(p),N_2(p),N_3(p),N_5(p)$ explicitly in terms of $n$ and $p$ without using fractional parts.
After taking a suitable integer $\lambda$, the right-hand side of \eqref{4.11} then gives a lower bound that we need:
\begin{table}[htbp]
\centering
\renewcommand{\arraystretch}{1.8}
\resizebox{\textwidth}{!}
{\begin{tabular}{@{}ccccccc@{}}
\toprule
			Range of $p$
			& $N_0(p)$
			& $N_2(p)$
			& $N_3(p)$
			& $N_5(p)$
			& $\lambda$
			& RHS of \eqref{4.11}
			\\
			\midrule
			$n<p\le \dfrac{8n}{7}$
			& $2p-2n$
			& $0$
			& $8n-7p+1$
			& $6p-6n-1$
			& $4$
			& $-8n+5p-5$
			\\
			$\dfrac{8n}{7}<p\le \dfrac{6n}{5}$
			& $6n-5p+1$
			& $7p-8n-1$
			& $0$
			& $2n-p$
			& $4$
			& $-32n+26p-8$
			\\
			$\dfrac{6n}{5}<p\le \dfrac{4n}{3}$
			& $2p-2n$
			& $0$
			& $8n-6p+1$
			& $5p-6n-1$
			& $2$
			& $-8n+p-3$
			\\
			$\dfrac{4n}{3}<p\le \dfrac{3n}{2}$
			& $6n-4p+1$
			& $6p-8n-1$
			& $0$
			& $2n-p$
			& $2$
			& $-32n+19p-6$
			\\
			$\dfrac{3n}{2}<p\le \dfrac{8n}{5}$
			& $2p-2n$
			& $0$
			& $8n-5p+1$
			& $4p-6n-1$
			& $0$
			& $-2n-5p-1$
			\\
			$\dfrac{8n}{5}<p\le 2n$
			& $6n-3p+1$
			& $5p-8n-1$
			& $0$
			& $2n-p$
			& $0$
			& $-26n+10p-4$
			\\
			$2n<p\le \dfrac{8n}{3}$
			& $p-2n$
			& $0$
			& $8n-3p+1$
			& $3p-6n-1$
			& $1$
			& $-n-p-2$
			\\
			$\dfrac{8n}{3}<p\le 3n$
			& $6n-2p+1$
			& $3p-8n-1$
			& $0$
			& $2n$
			& $1$
			& $-25n+8p-5$
			\\
			\bottomrule
		\end{tabular}%
	}
\caption{Explicit values of the right-hand side of \eqref{4.11} in each range of $p$}
\label{Table1}
\end{table}

Finally, by \eqref{4.11}, the last column of Table \ref{Table1}, and Lemma \ref{lem_4.1}, we obtain
\begin{align*}
&\ \liminf_{n \to \infty} \frac{1}{n^2}\sum_{n < p \le 3n} v_p(J_n) \log p \\
\ge&\ \int_{1}^{8/7} (-8+5x) \,\rd x + \int_{8/7}^{6/5} (-32+26x) \,\rd x + \cdots + \int_{8/3}^{3} (-25+8x)\,\rd x \\
=&\ -\frac{657}{70} > -9.39.
\end{align*}
\end{proof}

At last, we address the most important range $3n<p \le 6n$.

\begin{lemma}\label{lem_4.4}
We have
\[
\liminf_{n \to \infty} \frac{1}{n^2}\sum_{3n < p \le 6n} v_p(J_n) \log p \ge -15.75.
\]
\end{lemma}

\begin{proof}
First, fix a prime $p$ in the range $3n < p \le 6n$. 
We assume $n \ge 10$ so that $p>3n>\sqrt{8n}$.
Since $p > 3n$, we have
\[
- \frac{9n^2}{p} + 3n -p \left\{ \frac{3n}{p} \right\}\left( 1 - \left\{ \frac{3n}{p} \right\} \right) =0.
\]
Thus, Lemma \ref{lem_vpJn} reduces to
\begin{equation}\label{4.12}
v_p(J_n) \ge \min_{\substack{ 0\le m_\rho\le 2+2\lfloor (6n-\rho)/p \rfloor \\ \sum_{\rho=0}^{p-1}m_\rho=3n}} \sum_{\rho=0}^{p-1}F^{\rho \bmod p}(m_\rho).
\end{equation}
We can determine the full list of $F^{\rho \bmod p}(m)$ ($0 \le m \le 2+2\lfloor (6n-\rho)/p \rfloor$) as follows.

Since $3n < p \le 6n$, we have
\[
\left\lfloor \frac{6n}{p} \right\rfloor = 1, \quad \left\lfloor \frac{6n-\rho}{p} \right\rfloor = \mathbf{1}_{\rho \le 6n - p}, \quad \left\lfloor \frac{2n+\rho}{p} \right\rfloor= \mathbf{1}_{\rho \ge p-2n},\quad \left\lfloor \frac{2n}{p} \right\rfloor = 0.
\]
Therefore, \eqref{def_F1} reduces to
\begin{align}
F^{\rho \bmod p}(m) =&\ m^2 + \left( -2 \cdot \mathbf{1}_{\rho \le 6n - p} + 3 \cdot \mathbf{1}_{\rho \ge p-2n} - 3 \right) m \notag\\
&+ 3 \cdot\max\left( 0, m - 2\cdot \mathbf{1}_{\rho \le 6n - p} \right) + \mathbf{1}_{m=1} \cdot \mathbf{1}_{\rho > 6n - p} \cdot \mathbf{1}_{\rho < p-2n}. \label{4.13}
\end{align}

\textbf{Case 1}: $3n < p \le 7n/2$. 
By \eqref{4.13}, we have
\begin{equation}\label{4.14}
F^{\rho \bmod p}(m) = \begin{cases}
m^2 -5m +3\cdot\max(0,m-2) \quad(0 \le m \le 4) \quad&\text{if } 0 \le \rho < p -2n \\
m^2 -2m + 3\cdot\max(0,m-2) \quad(0 \le m \le 4)  \quad&\text{if } p-2n \le \rho \le 6n -p \\
m^2 + 3m \quad(0 \le m \le 2)  \quad&\text{if } 6n - p < \rho < p
\end{cases}.
\end{equation}
A finite check shows
\[
F^{\rho \bmod p}(m) \ge -m -4 \cdot \mathbf{1}_{\rho < p-2n} \quad (0 \le m \le 2+2\lfloor (6n-\rho)/p\rfloor,\ 0 \le \rho \le p-1).
\]
Therefore, \eqref{4.12} implies
\[
v_p(J_n) \ge -3n -4(p-2n) = 5n - 4p.
\]
Then, Lemma \ref{lem_4.1} gives
\begin{equation}\label{case1}
\liminf_{n \to \infty} \frac{1}{n^2}\sum_{3n < p \le 7n/2} v_p(J_n) \log p \ge \int_3^{7/2} (5-4x)\,\rd x = -4.
\end{equation}

\textbf{Case 2}: $7n/2 < p \le 4n$.
The finite list \eqref{4.14} still holds in this case.
One can check  
\[
F^{\rho \bmod p}(m) \ge -2m -2 \cdot \mathbf{1}_{\rho < p-2n} \quad (0 \le m \le 2+2\lfloor (6n-\rho)/p \rfloor,\ 0 \le \rho \le p-1).
\]
Therefore, \eqref{4.12} implies
\[
v_p(J_n) \ge -6n -2(p-2n) = -2n - 2p.
\]
Then, Lemma \ref{lem_4.1} gives
\begin{equation}\label{case2}
\liminf_{n \to \infty} \frac{1}{n^2}\sum_{7n/2 < p \le 4n} v_p(J_n) \log p \ge \int_{7/2}^{4} (-2-2x)\,\rd x = -\frac{19}{4}.
\end{equation}

\textbf{Case 3}: $4n < p \le 9n/2$.
By \eqref{4.13}, we have 
\begin{equation}\label{4.17}
F^{\rho \bmod p}(m) = \begin{cases}
m^2 -5m +3\cdot\max(0,m-2) \quad(0 \le m \le 4) \quad&\text{if } 0 \le \rho \le 6n-p \\
m^2 + \mathbf{1}_{m=1} \quad(0 \le m \le 2)  \quad&\text{if } 6n-p < \rho < p-2n \\
m^2 + 3m \quad(0 \le m \le 2)  \quad&\text{if } p-2n \le \rho < p
\end{cases}.
\end{equation}
A finite check shows
\[
F^{\rho \bmod p}(m) \ge -2m -2 \cdot \mathbf{1}_{\rho \le 6n-p} \quad (0 \le m \le 2+2\lfloor (6n-\rho)/p \rfloor,\ 0 \le \rho \le p-1).
\]
Therefore, \eqref{4.12} implies
\[
v_p(J_n) \ge -6n -2(6n-p+1) = -18n +2p -2.
\]
Then, Lemma \ref{lem_4.1} gives
\begin{equation}\label{case3}
\liminf_{n \to \infty} \frac{1}{n^2}\sum_{4n < p \le 9n/2} v_p(J_n) \log p \ge \int_{4}^{9/2} (-18+2x)\,\rd x = -\frac{19}{4}.
\end{equation}

\textbf{Case 4}: $9n/2 < p \le 6n$. The finite list \eqref{4.17} still holds in this case. 
One can check
\[
F^{\rho \bmod p}(m) \ge 2m -10 \cdot \mathbf{1}_{\rho \le 6n -p} \quad (0 \le m \le 2+2\lfloor (6n-\rho)/p\rfloor,\ 0 \le \rho \le p-1).
\]
Therefore, \eqref{4.12} implies
\[
v_p(J_n) \ge 6n -10(6n-p+1) = -54n + 10p - 10.
\]
Then, Lemma \ref{lem_4.1} gives
\begin{equation}\label{case4}
\liminf_{n \to \infty} \frac{1}{n^2}\sum_{9n/2 < p \le 6n} v_p(J_n) \log p \ge \int_{9/2}^{6} (-54 +10x)\,\rd x = -\frac{9}{4}.
\end{equation}

Finally, summing up \eqref{case1}, \eqref{case2}, \eqref{case3}, and \eqref{case4}, we obtain
\[
\liminf_{n \to \infty} \frac{1}{n^2}\sum_{3n < p \le 6n} v_p(J_n) \log p \ge -\frac{63}{4} = -15.75.
\]
\end{proof}

Putting Lemmas \ref{lem_4.2}--\ref{lem_4.4} and Lemma \ref{lem_cla_p_est} (3) together, we obtain the following.

\begin{proposition}\label{prop_p_est}
We have 
\[
\liminf_{n \to \infty} \frac{1}{n^2}\sum_{p} v_p(J_n)\log p \ge -30.83.
\]
In particular, all sufficiently large $n$ with $J_n \neq 0$ satisfy $|J_n| \ge \exp\left( -30.84 \cdot n^2 \right)$.
\end{proposition}

\section{Real estimates for the determinants}\label{sec_6}

In this section, we estimate $|J_n|$ as a real number.
Recall that the linear operator $\mathcal{L}_n$ is defined by \eqref{def_Lambda}.
We first express it as a contour integral. 
This is a classical argument. 
Such integrals have been studied extensively by Nesterenko \cite{Nesterenko1996}, Rivoal \cite{Rivoal2002}, and Zudilin \cite{Zudilin2001,Zudilin2002a,Zudilin2002b,Zudilin2004}.

\begin{lemma}\label{lem_6.1}
For any $P \in \mathbb{C}[t]_{<6n}$, 
\[
\mathcal{L}_n(P) = \frac{1}{2\pi\bi}\int_{2n+\frac{1}{2}-\bi\infty}^{2n+\frac{1}{2}+\bi\infty} R_n(z)P(z) \left( \frac{\pi}{\sin \pi z} \right)^2 \Big( 3a - 2b \pi \cot(\pi z) \Big) \rd z.
\]
\end{lemma}

\begin{proof}
This is a standard application of Cauchy's residue theorem. 
For any sufficiently large integer $T>10n$, consider the clockwise rectangular contour $C_{T}$, whose four vertices are $2n+1/2 \pm \bi T$ and $T+1/2 \pm \bi T$.

On the two horizontal edges of $C_T$, since $\deg R_nP \le -3$, we have (as $T \to +\infty$)
\begin{align*}
\left| R_n(z)P(z) \right| &= O(T^{-3}), \\ 
\left|  \left( \frac{\pi}{\sin \pi z} \right)^2 \Big( 3a - 2b \pi \cot(\pi z) \Big)  \right| &\le 10^3(|a|+|b|) e^{-\pi T}.
\end{align*}
On the vertical edge $z=T+1/2 + \bi y$ ($-T \le y \le T$), since $T \in \bbZ$, we have
\[
\left|  \left( \frac{\pi}{\sin \pi z} \right)^2 \Big( 3a - 2b \pi \cot(\pi z) \Big)  \right| \le 10^3(|a|+|b|) e^{-2\pi |y|}.
\]
Using Cauchy's residue theorem for $C_T$ and letting $T \to +\infty$, we obtain
\begin{align}
&\ \frac{1}{2\pi\bi}\int_{2n+\frac{1}{2}-\bi\infty}^{2n+\frac{1}{2}+\bi\infty} R_n(z)P(z) \left( \frac{\pi}{\sin \pi z} \right)^2 \Big( 3a - 2b \pi \cot(\pi z) \Big) \rd z \notag\\
=&\ - \sum_{m= 2n+1}^{\infty} \operatorname{Res}_{z=m} \left( R_n(z)P(z) \left( \frac{\pi}{\sin \pi z} \right)^2 \Big( 3a - 2b \pi \cot(\pi z) \Big) \right). \label{6.1}
\end{align}

At any $m \in \bbZ_{\ge 2n+1}$, we have 
\begin{align*}
\pi\cot(\pi z) &= \frac{1}{(z-m)} - \frac{\pi^2}{3}(z-m) + O((z-m)^3),  \\
\left( \frac{\pi}{\sin \pi z} \right)^2 &= \frac{1}{(z-m)^2} + \frac{\pi^2}{3} + O((z-m)^2), \\
R_n(z)P(z) &= (R_nP)(m) + (R_nP)'(m) \cdot (z-m) + \frac{1}{2} (R_nP)''(m) \cdot (z-m)^2 + O((z-m)^3).
\end{align*}
Therefore, 
\begin{align}
&\ \operatorname{Res}_{z=m} \left( R_n(z)P(z) \left( \frac{\pi}{\sin \pi z} \right)^2 \Big( 3a - 2b \pi \cot(\pi z) \Big) \right) \notag\\
=&\ 3a \cdot (R_nP)'(m) -b \cdot (R_nP)''(m). \label{6.2}
\end{align}
Substituting \eqref{6.2} into \eqref{6.1}, we obtain 
\begin{align*}
&\ \frac{1}{2\pi\bi}\int_{2n+\frac{1}{2}-\bi\infty}^{2n+\frac{1}{2}+\bi\infty} R_n(z)P(z) \left( \frac{\pi}{\sin \pi z} \right)^2 \Big( 3a - 2b \pi \cot(\pi z) \Big) \rd z \notag\\
=&\ -3a \sum_{m=2n+1}^{\infty} (R_nP)'(m) + b \sum_{m=2n+1}^{\infty} (R_nP)''(m).
\end{align*}
By \eqref{def_Rn(t)}, we have $(R_nP)'(m)=(R_nP)''(m)=0$ for $m \in \{1,2,\ldots,2n\}$. So
\begin{align*}
&\ \frac{1}{2\pi\bi}\int_{2n+\frac{1}{2}-\bi\infty}^{2n+\frac{1}{2}+\bi\infty} R_n(z)P(z) \left( \frac{\pi}{\sin \pi z} \right)^2 \Big( 3a - 2b \pi \cot(\pi z) \Big) \rd z \notag\\
=&\ -3a \sum_{m=1}^{\infty} (R_nP)'(m) + b \sum_{m=1}^{\infty} (R_nP)''(m) = \mathcal{L}_n(P).
\end{align*}
The last step used \eqref{def_Lambda}.
\end{proof}

Next, we estimate the rational function $R_n(z)$ on the vertical line $\operatorname{Re} z = 2n+1/2$.

\begin{lemma}\label{lem_6.2}
For any $y \in \mathbb{R}$, we have
\[
\left| R_n\left( 2n+\frac{1}{2} +\bi ny \right) \right| \le \frac{10}{n} \cdot  e^{nG(y)},
\]
where the function $G(y)$ is given by
\begin{equation}\label{def_G}
G(y) \coloneq 5\log(y^2+4) - 8\log(y^2+64) +5y\arctan \frac{2}{y} -2y \arctan\frac{8}{y} + 12\log 6 -6 \log 2.
\end{equation}
The value $G(0)$ is understood as $\lim_{y \to 0} G(y)$.
\end{lemma}  

\begin{proof}
By symmetry and continuity, we may fix $y>0$.
By \eqref{def_Rn(t)}, we have
\begin{align}
&\ \log \left| R_n\left( 2n+\frac{1}{2} +\bi ny \right) \right| \notag\\
=&\ \log \frac{(6n)!^2}{(2n)!^3} + \frac{3}{2}\sum_{k=1}^{2n} \log\left( n^2y^2 + \left(k-\frac{1}{2}\right)^2 \right) - \sum_{j=0}^{6n}\log\left( n^2y^2 + \left(2n+j+\frac{1}{2}\right)^2 \right) \notag\\
=&\ \log \frac{(6n)!^2}{n^{6n+2} (2n)!^3} + \frac{3}{2} \sum_{k=1}^{2n} g\left( \frac{k-\frac{1}{2}}{n} \right) - \sum_{j=0}^{6n} g\left( 2 + \frac{j+\frac{1}{2}}{n} \right),  \label{6.4}
\end{align}
where the function $g$ is given by
\[
g(x) \coloneq \log(x^2+y^2), \quad x \in \mathbb{R}. 
\]
It is natural to compare the sum of values of $g$ with its integral.
Note that 
\[
x\log(x^2+y^2) -2x + 2y\arctan \frac{x}{y}
\]
is a primitive function of $g(x)$. 
Since $g(x)$ is increasing on $(0,+\infty)$, we obtain
\begin{align}
\sum_{k=1}^{2n} g\left( \frac{k-\frac{1}{2}}{n} \right) &\le \frac{1}{2} g\left( 2 -\frac{1}{2n}\right) + n\int_{\frac{1}{2n}}^{2} g(x)\,\rd x \le \frac{1}{2}g(2) + n\int_{\frac{1}{2n}}^{2} g(x)\,\rd x \notag\\
&= \left( 2\log(y^2+4) - 4 + 2y\arctan \frac{2}{y}\right) n + \frac{1}{2}\log \frac{y^2+4}{y^2+\frac{1}{4n^2}} + 1 -2ny \arctan \frac{1}{2ny} \notag\\
&\le \left( 2\log(y^2+4) - 4 + 2y\arctan \frac{2}{y}\right) n + \log(4n) + 1, \label{6.5}
\end{align}
and
\begin{align}
\sum_{j=0}^{6n} g\left( 2 + \frac{j+\frac{1}{2}}{n} \right)  &\ge \frac{1}{2}g\left( 2+\frac{1}{2n} \right) + \frac{1}{2}g\left( 8+\frac{1}{2n} \right)+ n \int_{2}^{8} g(x)\,\rd x    \notag\\ 
&> \left( 8\log(y^2+64) -  2\log(y^2+4) + 2y\arctan\frac{8}{y} -2y\arctan \frac{2}{y} -12 \right)n \notag \\
&\quad +\log 2 + \log 8. \label{6.6}
\end{align}
Substituting \eqref{6.5} and \eqref{6.6} into \eqref{6.4}, and using the well-known inequality \cite{Robbins1955}: 
\[
\sqrt{2\pi m} \left(\frac{m}{e}\right)^m \le m! \le \sqrt{2\pi m} \left(\frac{m}{e}\right)^m e^{\frac{1}{12m}}, \quad m\in \bbZ_{>0},
\]
we easily complete the proof.
\end{proof}

\begin{lemma}\label{lem_6.3}
For any $z=2n+1/2+\bi x$ with $x \in \mathbb{R}$, we have
\[
\left| \frac{1}{2\pi} R_n(z) \left( \frac{\pi}{\sin \pi z} \right)^2 \Big( 3a - 2b \pi \cot(\pi z) \Big)  \right| \le 10^{3}\left( |a|+|b| \right) n^{-1} \cdot e^{c_1 n} e^{-c_2 |x|},
\]
where 
\begin{equation}\label{def_c1c2}
c_1 = G\left(\frac{4}{5}\right) - \frac{4}{5} G'\left( \frac{4}{5} \right) < 0, \qquad c_2 = 2\pi - G'\left( \frac{4}{5} \right)>0.
\end{equation}
The function $G(y)$ is given by \eqref{def_G}.
\end{lemma}

\begin{proof}
We have 
\begin{align*}
\left| \frac{\pi}{\sin \pi z} \right| &= \frac{2\pi}{e^{\pi x} + e^{-\pi x}} \le 2\pi e^{-\pi|x|}, \\
\left|  \cot \pi z  \right| &= \left| \frac{e^{\pi x} - e^{-\pi x}}{e^{\pi x} + e^{-\pi x}} \right| \le 1.
\end{align*}
By Lemma \ref{lem_6.2} with $y = x/n$, we have
\[
|R_n(z)| \le \frac{10}{n} e^{nG(x/n)} = \frac{10}{n} e^{nG(|x|/n)},
\]
where the function $G$ is given by \eqref{def_G}. 
Putting the above estimates together, we obtain
\begin{equation}\label{6.8}
\left| \frac{1}{2\pi} R_n(z) \left( \frac{\pi}{\sin \pi z} \right)^2 \Big( 3a - 2b \pi \cot(\pi z) \Big)  \right| \le 10^{3}\left(|a|+|b|\right) n^{-1} \cdot \exp\left( -2\pi |x| + n G\left(\frac{|x|}{n}\right) \right).
\end{equation}

Now, differentiating \eqref{def_G} gives
\[
G'(y) = 5 \arctan \frac{2}{y} - 2 \arctan \frac{8}{y}, \quad G''(y)= \frac{6(y^2 - 96)}{(y^2+4)(y^2+64)}, \quad y \in (0,+\infty).
\]
Let 
\[
\lambda_0 = G'\left(\frac{4}{5}\right) = 5 \arctan \frac{5}{2} - 2 \arctan 10 = 3.009194\ldots.
\]
The function $G'(y) - \lambda_0$ is decreasing on $(0,\sqrt{96}]$ and increasing on $[\sqrt{96},+\infty)$ with $\lim_{y \to +\infty} (G'(y)-\lambda_0) = -\lambda_0 < 0$. 
Therefore, $G'(y)-\lambda_0$ is positive on $(0,4/5)$ and negative on $(4/5,+\infty)$.
So 
\[
G(y) - \lambda_0 y \le G\left( \frac{4}{5} \right) - \frac{4}{5}\lambda_0 = c_1, \quad \text{for all } y \in [0,+\infty).
\]
Thus,
\[
n G\left( \frac{|x|}{n} \right) \le c_1 n + \lambda_0 |x|.
\]
Substituting this into \eqref{6.8} completes the proof.
\end{proof}

\begin{lemma}\label{lem_6.4}
For any integer $n \ge 2 \cdot 10^{3}(|a|+|b|)$, we have
\[
|J_n| \le \frac{e^{3c_1 n^2}}{2^{3n}(3n)!\prod_{k=0}^{3n-1}k!^2} \int_{\mathbb{R}^{3n}} \prod_{1 \le i < j  \le 3n} (x_i - x_j)^2 \prod_{\ell = 1}^{3n} e^{-c_2|x_{\ell}|} \,\rd x_\ell,
\]
where $c_1,c_2$ are given by \eqref{def_c1c2}.
\end{lemma}

\begin{proof}
By \eqref{def_Jn} and Lemma \ref{lem_6.1}, we have
\[
J_n = \frac{1}{\prod_{k=0}^{3n-1}k!^2} \det \left( \int_{-\infty}^{+\infty} \left(2n+\frac{1}{2} + \bi x\right)^{i+j} \omega(x) \,\rd x \right)_{0 \le i,j < 3n},
\]
where 
\[
\omega(x) = \frac{1}{2\pi} R_n(z) \left( \frac{\pi}{\sin \pi z} \right)^2 \Big( 3a - 2b \pi \cot(\pi z) \Big)  \Big|_{z = 2n+1/2+\bi x}.
\]
By Heine's identity \cite{Heine1878} (which follows from determinant expansion and Fubini's theorem; see Zudilin \cite[p. 304]{Zudilin2017}), we obtain
\[
J_n = \frac{1}{\prod_{k=0}^{3n-1}k!^2} \cdot \frac{\bi^{3n(3n-1)}}{(3n)!} \int_{\mathbb{R}^{3n}} \prod_{1 \le i < j  \le 3n} (x_i - x_j)^2 \prod_{\ell = 1}^{3n} \omega(x_\ell)\,\rd x_\ell.
\]
By Lemma \ref{lem_6.3}, 
\[
|J_n| \le \left( \frac{10^3(|a|+|b|)}{n} \right)^{3n} \cdot \frac{e^{3c_1 n^2}}{(3n)!\prod_{k=0}^{3n-1}k!^2} \int_{\mathbb{R}^{3n}} \prod_{1 \le i < j  \le 3n} (x_i - x_j)^2 \prod_{\ell = 1}^{3n} e^{-c_2|x_{\ell}|} \,\rd x_\ell.
\]
This completes the proof.
\end{proof}

We need one last lemma, which is an evaluation of a multiple integral.

\begin{lemma}\label{lem_6.5}
We have
\[
\frac{1}{2^{3n}(3n)!\prod_{k=0}^{3n-1}k!^2} \int_{\mathbb{R}^{3n}} \prod_{1 \le i < j  \le 3n} (x_i - x_j)^2 \prod_{\ell = 1}^{3n} \frac{2}{e^{c_2x_\ell} + e^{-c_2x_\ell}} \,\rd x_\ell = \left( \frac{\pi}{2c_2} \right)^{9n^2}.
\]
\end{lemma}

\begin{proof}
Apply the change of variables $x_\ell = (\pi/c_2) y_{\ell}$, $\ell=1,2,\ldots,3n$. 
The assertion of this lemma is equivalent to 
\begin{equation}\label{6.9}
\frac{1}{(3n)!} \int_{\mathbb{R}^{3n}} \prod_{1 \le i < j  \le 3n} (y_i - y_j)^2 \prod_{\ell = 1}^{3n} \frac{2}{e^{\pi y_\ell} + e^{-\pi y_\ell}} \,\rd y_\ell = \prod_{k=0}^{3n-1} \frac{k!^2}{4^k}.
\end{equation}  
For the weight function $\omega(y)=2/(e^{\pi y}+e^{-\pi y})$ on $\mathbb{R}$, the corresponding orthogonal polynomials are known \cite{EMOT1953,Ismail2009,Meixner1934,Pollaczek1950} as Meixner--Pollaczek polynomials (with parameters $\lambda =1/2$ and $\phi=\pi/2$). 
By \cite[Equations (5.9.8) and (5.9.9)]{Ismail2009}, there exist polynomials $P_k(y) \in \bbQ[y]$ ($k \in \bbZ_{\ge 0}$) with $\deg P_k = k$ and leading coefficients $[y^k](P_k) = 2^k/k!$ such that
\[
\int_{\mathbb{R}} P_j(y) P_k(y) \frac{2}{e^{\pi y}+e^{-\pi y}} \,\rd y = \delta_{j,k}, \quad j,k \in \bbZ_{\ge 0}.
\]
Then, by Heine's identity again,
\[
\frac{1}{(3n)!} \int_{\mathbb{R}^{3n}} \det\Big( P_k(y_{j+1}) \Big)_{0 \le j,k < 3n}^2 \prod_{\ell = 1}^{3n} \frac{2}{e^{\pi y_\ell} + e^{-\pi y_\ell}} \,\rd y_\ell = 1.
\]
Since the transfer matrix from $1,y,y^2,\ldots,y^{3n-1}$ to $P_0(y),P_1(y),\ldots,P_{3n-1}(y)$ is triangular with diagonal entries $2^k/k!$ ($k=0,1,\ldots,3n-1$), we have
\[
\det\Big( P_k(y_{j+1}) \Big)_{0 \le j,k < 3n} = \prod_{k=0}^{3n-1} \frac{2^k}{k!} \cdot \prod_{1 \le i < j  \le 3n} (y_j - y_i),
\]
which proves \eqref{6.9}.
\end{proof}

\begin{proposition}\label{prop_real}
For any integer $n \ge 2 \cdot 10^{3}(|a|+|b|)$, we have 
\[
|J_n| \le e^{ -31.61 \cdot n^2}.
\]
\end{proposition}

\begin{proof}
Since $c_2 > 0$ (see \eqref{def_c1c2}), we have
\[
e^{-c_2|x|} \le \frac{2}{e^{c_2x} + e^{-c_2x}}, \quad x \in \mathbb{R}.
\]
Substituting this into Lemma \ref{lem_6.4} and then using Lemma \ref{lem_6.5}, we obtain
\[
|J_n| \le e^{ - c_3 n^2}, \quad \text{for any } n \ge 2 \cdot 10^{3}(|a|+|b|),
\]
where 
\[
c_3 = 9 \log \frac{2c_2}{\pi} - 3c_1,
\]
and $c_1,c_2$ are given by \eqref{def_c1c2}.
Explicitly,
\begin{align*}
c_3 &= 9 \log \left(4+\frac{4\arctan 10-10 \arctan (5/2)}{\pi}\right) -3\left(5 \log \frac{116}{25}-8 \log \frac{1616}{25} + 12 \log 6 - 6\log 2\right) \\
&= 31.614437\ldots.
\end{align*}
\end{proof}

Now, Propositions \ref{prop_nonva}, \ref{prop_p_est}, and \ref{prop_real} together give a contradiction. 
So the assumption \eqref{def_abc} cannot hold.
This proves Theorem \ref{thm_main}.
We conclude with a remark.
\begin{remark}
One can show a better real estimate by shifting the contour and using the equilibrium measure (see \cite{Saff2010}).
For the $p$-adic estimate in Proposition \ref{prop_p_est}, by using sharper estimates in the summation over primes and a more detailed case analysis, one can deduce from Lemma \ref{lem_vpJn} a better arithmetic estimate.
Therefore, the margin in the main theorem is slightly larger than what we proved. 
Our approach is mainly intended to keep the proof as simple as possible.
\end{remark}

\vspace*{3mm}
\begin{flushright}
\begin{minipage}{158mm}\sc\footnotesize
School of Mathematical Sciences, Xiamen University, Fujian, China \\
{\it E-mail addresses}: \href{mailto:lilaimath@gmail.com}{{\tt lilaimath@gmail.com}}, \href{mailto:lilai@xmu.edu.cn}{{\tt lilai@xmu.edu.cn}} 
\vspace*{3mm}
\end{minipage}
\end{flushright}

\end{document}